\documentclass[oneside,english]{amsart}
\usepackage[T1]{fontenc}
\usepackage[latin9]{inputenc}
\usepackage{amstext}
\usepackage{amsthm}
\usepackage{amssymb}

\makeatletter
\numberwithin{equation}{section}
\numberwithin{figure}{section}
\theoremstyle{plain}
\newtheorem{thm}{\protect\theoremname}

\numberwithin{thm}{section}
\usepackage{hyperref}
\usepackage{xcolor}

\makeatother

\usepackage{babel}
\providecommand{\theoremname}{Theorem}

\begin{document}
\title{Upgrading subordination properties in free probability}
\author{Serban T. Belinschi and Hari Bercovici}
\begin{abstract}
The existence of Voiculescu's subordination functions in the context of non-tracial operator-valued $C^*$-probability spaces has been established using analytic function theory methods.
We use a matrix construction to show that the subordination functions thus obtained also satisfy an appropriately modified form of subordination for conditional expectations.
\end{abstract}

\maketitle

\section{Introduction\label{sec:Introduction}}

Suppose that $(\mathcal{A},\tau)$ is a tracial $W^{*}$-probability
space. That is, $\mathcal{A}$ is a von Neumann algebra, and $\tau:\mathcal{A}\to\mathbb{C}$
is a faithful normal trace mapping the unit of $\mathcal A$ to the complex number one. 
Let $x_{1},x_{2}\in\mathcal{A}$ be two
freely independent, selfadjoint elements. Voiculescu \cite{entropy-1}
showed that, at least generically, there exists an analytic selfmap
$\omega$ of the complex upper half-plane $\mathbb{H}$ satisfying
the identity
\begin{equation}
\tau\left((\lambda-(x_{1}+x_{2}))^{-1}\right)=\tau\left((\omega(\lambda)-x_{1})^{-1}\right),\quad\lambda\in\mathbb{H}.\label{eq:v-subordination}
\end{equation}
The genericity hypothesis was subsequently removed by Biane \cite{biane},
who showed that $\mathcal{\omega}$ satisfies a stronger condition
than (\ref{eq:v-subordination}). In order to state this condition,
we denote by $\mathcal{Y}$ the von Neumann algebra generated by $x_1$,
that is, $\mathcal{Y}=\{x_1\}''$, and we let $\mathbb{E}_{\mathcal{Y}}:\mathcal{A}\to\mathcal{Y}$
be the trace-preserving conditional expectation. Then
\begin{equation}
\mathbb{E}_{\mathcal{Y}}[(\lambda-(x_{1}+x_{2}))^{-1}]=(\omega(\lambda)-x_{1})^{-1},\quad\lambda\in\mathbb{H},\label{eq:Biane subordination}
\end{equation}
and this, of course, implies (\ref{eq:v-subordination}) upon applying
$\tau$. Since $\mathbb{E}_{\mathcal{Y}}$ is just the orthogonal
projection in the scalar product induced by $\tau$, (\ref{eq:Biane subordination})
is equivalent to the following relation that involves just $\tau$:
\[
\tau\left(y(\lambda-(x_{1}+x_{2}))^{-1}\right)=\tau\left(y(\omega(\lambda)-x_{1})^{-1}\right),\quad\lambda\in\mathbb{H},y\in\mathcal{Y}.
\]
 Subsequently, Voiculescu \cite{coalgebra} discovered an underlying algebraic
structure that allows for a conceptually simple proof of Biane's extension
and, indeed, for a much more general result. Namely, suppose that
in addition to the probability space $(\mathcal{A},\tau)$, we consider
a von Neumann subalgebra $\mathcal{B}\subset\mathcal{A}$ containing
the unit of $\mathcal{A}$, and denote by $\mathbb{E}_{\mathcal{B}}:\mathcal{A}\to\mathcal{B}$
the trace-preserving conditional expectation. Let $x_{1},x_{2}\in\mathcal{A}$
be two selfadjoint elements of $\mathcal{A}$ that are free relative
to $\mathbb{E_{\mathcal{B}}}$ (or, more simply, $\mathbb{E}_{\mathcal{B}}$-free).
Denote by $\mathbb{H}(\mathcal{B})$ the collection of those $b\in\mathcal{B}$
that satisfy $\Im b:=\frac{b-b^*}{2i}\ge\varepsilon$ for some $\varepsilon>0$. Finally,
let $\mathcal{B}\langle x_{1}\rangle$ be the von Neumann algebra
generated by $\mathcal{B}$ and $x_{1}$, and let $\mathbb{E}_{\mathcal{B}\langle x_{1}\rangle}:\mathcal{A}\to\mathcal{B}\langle x_{1}\rangle$
denote the trace-preserving conditional expectation. Then there exists
an analytic function $\omega:\mathbb{H}(\mathcal{B})\to\mathbb{H}(\mathcal{B})$
such that
\begin{equation}
\mathbb{E}_{\mathcal{B}\langle x_{1}\rangle}\left[(b-(x_{1}+x_{2}))^{-1}\right]=(\omega(b)-x_{1})^{-1},\quad b\in\mathcal{B}.\label{eq:V subordination ope-valued}
\end{equation}
Like (\ref{eq:Biane subordination}), this identity can be rewritten as
\begin{equation}
\tau\left(y(b-(x_1+x_2))^{-1}\right)=\tau\left(y(\omega(b)-x_1)^{-1}\right),\quad b\in\mathcal{B},y\in\mathcal{B}\langle x_{1}\rangle.\label{eq:op-val subordination written with tau}
\end{equation}
 Using the fact that $\mathbb{E}_{\mathcal{B}}$ is trace-preserving,
(\ref{eq:op-val subordination written with tau}) is also seen to
be equivalent to
\begin{equation}
\mathbb{E}_{\mathcal{B}}\left[y(b-(x_{1}+x_{2}))^{-1}\right]=\mathbb{E}_{\mathcal{B}}\left[y(\omega(b)-x_{1})^{-1}\right],\quad b\in\mathcal{B},y\in\mathcal{B}\langle x_{1}\rangle.\label{eq:equation to generalize}
\end{equation}
As seen in \cite{be-ma-spe}, the existence of subordination functions can
be proved using methods of Banach space analyticity. In fact \cite{be-ma-spe}
provides the following analog of Voiculescu's original result from
\cite{entropy-1} that does not require the presence of a trace and it applies
to operator-valued $C^{*}$-probability spaces. The statement uses
the notation 
\[
F_{x}(b)=\left(\mathbb{E}\left[(b-x)^{-1}\right]\right)^{-1}.
\]

\begin{thm}
\label{thm:BMS result} Let $\mathcal{A}$ be a unital $C^{*}$-algebra,
let $\mathcal{B}\subset\mathcal{A}$ be a unital sub-$C^{*}$-algebra
containing the unit of $\mathcal{A}$, and let $\mathbb{E}:\mathcal{A}\to\mathcal{B}$
be a faithful, completely positive conditional expectation. Given
selfadjoint elements $x_{1},x_{2}\in\mathcal{A}$ that are $\mathbb{E}$-free,
there exist unique analytic functions $\omega_{1},\omega_{2}:\mathbb{H}(\mathcal{B})\to\mathbb{H}(\mathcal{B})$
such that
\[
F_{x_{1}+x_{2}}(b)=F_{x_{1}}(\omega_{1}(b))=F_{x_{2}}(\omega_{2}(b))=\omega_{1}(b)+\omega_{2}(b)-b
\]
for every $b\in\mathbb{H}(\mathcal{B})$.
\end{thm}

The purpose of this note is to extend the subordination result of
\cite{coalgebra} to this general context. Of course, there may not exist
a conditional expectation onto $\mathcal{B}\langle x_{1}\rangle$,
so (\ref{eq:V subordination ope-valued}) may not make sense, but
(\ref{eq:equation to generalize}) does make sense, and it is this
equation that we extend to the general context in Theorem \ref{thm:E(Rc)}.
Note that this result is somewhat stronger than that of \cite{coalgebra}
in the tracial case because we allow for the conditional expectation
onto an algebra that is possibly larger than $\mathcal{B}\langle x_{1}\rangle$.
This improvement can however be obtained already with the methods
of \cite{coalgebra}. We actually prove an even more general result (Theorem
\ref{thm:main Banach algebra theorem}) that applies in the original
context of operator-valued free probability \cite{operator-val}.
In this result, $\mathcal{A}$ is just a Banach algebra and $\mathbb{E}:\mathcal{A}\to\mathcal{B}$
is a continuous conditional expectation. In this general context,
the subordination functions $\omega_{j}(b)$ are only defined for
invertible elements $b\in\mathcal{B}$ for which $\|b^{-1}\|$ is
sufficiently small. Theorem \ref{thm:E(Rc)} follows from the Banach
algebra result by analytic continuation.

The idea for this work arose from the observation that, in the context of the study of bi-free additive convolution \cite{BBGS}, a method based on analytic functions and expansion to operator matrices allows for the recovery of Biane's result. This led to the extension of the operator-valued subordination result of \cite{coalgebra} to the more general context described in Theorem \ref{thm:E(Rc)}, again using purely analytic elementary methods.

\section{Subordination in Banach algebraic probability spaces\label{sec:Subrdination-in-Banach}}

Let $(\mathcal{A},\mathcal{B},\mathbb{E})$ be a Banach algebraic
probability space. That is, $\mathcal{A}$ is a complex, unital Banach
algebra, $\mathcal{B}$ is a closed subalgebra containing the unit
of $\mathcal{A}$, and $\mathbb{E}:\mathcal{A\to B}$ is a continuous
conditional expectation. The concept of $\mathcal{B}$-\emph{freeness
with respect to} $\mathbb{E}$, which we refer to more simply as $\mathbb{E}$\emph{-freeness},
was introduced in \cite{operator-val}. It was also shown in \cite{operator-val} that the
calculation of (the symmetric parts of) free convolutions of $\mathcal{B}$-valued
distributions follows the same pattern first discovered in \cite{boxplus}
when $\mathcal{B}=\mathbb{C}$. We recall briefly the relevant notation.
Given a random variable, that is, an element $x\in\mathcal{A}$, and
given $b\in\mathcal{B},$ we write
\begin{align*}
G_{x}(b) & =\mathbb{E}\left[(b-x)^{-1}\right],\\
F_{x}(b) & =G_{x}(b)^{-1},\\
\widetilde{G}_{x}(b) & =G_{x}(b^{-1})=b+\sum_{n=1}^{\infty}\mathbb{E}\left[b(xb)^{n}\right],
\end{align*}
whenever the quantities on the right-hand side make sense. In particular,
$\widetilde{G}_{x}$ is defined for $\|b\|<1/\|x\|$ and it is an
analytic function. Moreover, the derivative of $\widetilde{G}_{x}$
at $b=0$ is the identity map of $\mathcal{B}.$ It follows that $\widetilde{G}_{x}$
maps conformally a neighborhood of $b=0$ onto another such neighborhood,
and thus has an inverse function denoted $\widetilde{K}_{x}$. Observe
also that
\[
\widetilde{G}_{x}(b)=b\left(1+\sum_{n=1}^\infty\mathbb{E}\left[(xb)^{n}\right]\right)=b(1+O(\|b\|)),
\]
and thus, provided $\|b\|$ is sufficiently small, $b$ is invertible
if and only if $\widetilde{G}_{x}(b)$ is invertible. It follows that
both $\widetilde{G}_{x}$ and $\widetilde{K}_{x}$ map invertible
elements to invertible elements when restricted to a sufficiently
small neighborhood of $0\in\mathcal{B}$. Finally, define
\[
R_{x}(b)=\widetilde{K}_{x}(b)^{-1}-b^{-1}
\]
for $b\in\mathcal{B}$ such that $\|b\|$ is sufficiently small. (As
in \cite{operator-val}, we reserve the exponent $-1$ for inverses in the algebra 
$\mathcal{A}.$) The function $R_{x}$ continues analytically to a
neighborhood of $b=0$.

Suppose now that $x_{1},x_{2}\in\mathcal{A}$ are $\mathbb{E}$-free,
and set $x=x_{1}+x_{2}$. Then it is shown in  \cite{operator-val} (see the end
of Section 4) that
\[
R_{x}(b)=R_{x_{1}}(b)+R_{x_{2}}(b)
\]
for $b\in\mathcal{B}$ such that $\|b\|$ is sufficiently small. We
define now 
\begin{equation}
\omega_{j}(b)=(\widetilde{K}_{x_{j}}(\widetilde{G}_{x}(b^{-1})))^{-1},\quad j=1,2,\label{eq:formula for omegas}
\end{equation}
for invertible elements $b\in\mathcal{B}$ such that $\|b^{-1}\|$
is sufficiently small. These functions are analytic for such values
of $b$, and the equalities
\begin{equation}
F_{x}(b)=F_{x_{1}}(\omega_{1}(b))=F_{x_{2}}(\omega_{2}(b))=\omega_{1}(b)+\omega_{2}(b)-b\label{eq:subordination eqn}
\end{equation}
hold throughout their domain of definition (see \cite{ber-voi-otaa} for the scalar-valued case). The functions $\omega_{j}$
will be referred to as the \emph{subordination functions} associated 
to the variables $x_{1}$ and $x_{2}$. (These subordination functions
are most useful when they can be defined on a standard domain, as
it happens when we work with selfadjoint variables in a $C^{*}$-probability
space---see Section \ref{sec: C* case} below.)

We are now ready to prove the main result of this section. The special
case $y=1$ follows, of course, from (\ref{eq:subordination eqn}).
\begin{thm}
\label{thm:main Banach algebra theorem}Let $(\mathcal{A},\mathcal{B},\mathbb{E})$
be a Banach algebraic probability space, and let $x_{1},x_{2}$ and
$y$ be elements of $\mathcal{A}$ such that $x_{2}$ is $\mathbb{E}$-free
from $\{x_{1},y\}$. Denote by $\omega_{1}$ and $\omega_{2}$ the
subordination functions associated to $x_{1}$ and $x_{2}$, and set
$x=x_{1}+x_{2}$. Then we have 
\[
\mathbb{E}[y(b-x)^{-1}]=\mathbb{E}[y(\omega_{1}(b)-x_{1})^{-1}]\ \text{ and }\ \mathbb{E}[(b-x)^{-1}y]=\mathbb{E}[(\omega_{1}(b)-x_{1})^{-1}y]
\]
for every invertible $b\in\mathcal{B}$ such that $\|b^{-1}\|$ is
sufficiently small.
\end{thm}

\begin{proof}
Consider the probability space $(\mathcal{A}_{2},\mathcal{B}_{2},\mathbb{E}_{2})$,
where $\mathcal{A}_{2}$ is the Banach algebra consisting of $2\times2$
matrices of the form
\[
\left[\begin{array}{cc}
a_{1,1} & 0\\
a_{2,1} & a_{2,2}
\end{array}\right],\quad a_{1,1},a_{2,1},a_{2,2}\in\mathcal{A},
\]
and $\mathcal{B}_{2}$ is the subalgebra consising of the matrices
whose entries belong to $\mathcal{B}.$ The norm of such a matrix
can be taken to be the operator norm when the matrix is viewed as
a left multiplication operator on $\mathcal{A}\oplus\mathcal{A}$
endowed with the $\ell^{1}$ norm. The conditional expectation $\mathbb{E}_{2}$
is simply $\mathbb{E}$ applied entrywise. The elements
\[
X_{1}=\left[\begin{array}{cc}
x_{1} & 0\\
y & 0
\end{array}\right],X_{2}=\left[\begin{array}{cc}
x_{1} & 0\\
0 & 0
\end{array}\right]
\]
 are $\mathbb{E}_{2}$-free (this is a straightforward verification,
but can as well be deduced from \cite{NSS} or from \cite[Corollary 3.7]{operator-val}). 
We denote by $X$ their sum, and we denote by $\Omega_{j}(B)$, $j=1,2$,
the corresponding subordination functions, defined for invertible
matrices $B\in\mathcal{B}_{2}$ such that $\|B^{-1}\|<\varepsilon$
for some $\varepsilon>0$. 

We first show that the matrices $\Omega_{j}(B)$ have the form
\[
\Omega_{j}(B)=\left[\begin{array}{cc}
\omega_{j}(b_{1,1}) & 0\\
g_{j}(B) & b_{2,2}
\end{array}\right],\quad B=\left[\begin{array}{cc}
b_{1,1} & 0\\
b_{2,1} & b_{2,2}
\end{array}\right],
\]
for some analytic functions $g_{1},g_{2}$. To do this, we use (\ref{eq:formula for omegas}),
so we calculate
\[
(B-X_{1})^{-1}=\left[\begin{array}{cc}
(b_{1,1}-x_{1})^{-1} & 0\\
-b_{2,2}^{-1}(b_{2,1}-y)(b_{1,1}-x_{1})^{-1} & b_{2,2}^{-1}
\end{array}\right],
\]
\begin{equation}
G_{X_{1}}(B)=\left[\begin{array}{cc}
G_{x_{1}}(b_{1,1}) & 0\\
-b_{2,2}^{-1}\mathbb{E}[(b_{2,1}-y)(b_{1,1}-x_{1})^{-1}] & b_{2,2}^{-1}
\end{array}\right],\label{eq:don't repeat too much}
\end{equation}
and
\[
\widetilde{G}_{X_{1}}(B)=G_{X_1}(B^{-1})=\left[\begin{array}{cc}
\widetilde{G}_{x_{1}}(b_{1,1}
) & 0\\
\mathbb{E}[(b_{2,1}b_{1,1}^{-1}+b_{2,2}y)(b_{1,1}^{-1}-x_1)^{-1}] & b_{2,2} 
\end{array}\right].
\]
The $(1,1)$ entry of this matrix only depends on $b_{1,1}$, and
therefore
\[
\widetilde{K}_{X_{1}}(B)=\left[\begin{array}{cc}
\widetilde{K}_{x_{1}}(b_{1,1}^{-1}) & 0\\
h_{1}(B) & b_{2,2}
\end{array}\right],
\]
for some analytic function $h_{1}$. Replacing $x_{1}$ by $x=x_{1}+x_{2}$,
we similarly obtain
\[
\widetilde{G}_{X}(B)=\left[\begin{array}{cc}
\widetilde{G}_{x}(b_{1,1}^{-1}) & 0\\
h(B) & b_{2,2}
\end{array}\right],
\]
and therefore, according to \eqref{eq:formula for omegas},
\begin{align*}
\Omega_{1}(B)=\widetilde{K}_{X_{1}}(\widetilde{G}_{X}(B^{-1}))^{-1}&=\left[\begin{array}{cc} 
\widetilde{K}_{x_{1}}(\widetilde{G}_{x}(b_{1,1}^{-1})) & 0\\
h_{1}(\widetilde{G}_{X}(B^{-1})) & b_{2,2}^{-1}
\end{array}\right]^{-1}\\ 
&=\left[\begin{array}{cc} \omega_1(b_{1,1}) & 0 \\ 
-b_{2,2}h_{1}(\widetilde{G}_{X}(B^{-1}))\omega_1(b_{1,1}) & b_{2,2} \end{array}\right]. 
\end{align*}
The argument for $\Omega_{2}$ is obtained upon replacing $x_{1}$
by $x_{2}$ and $y$ by zero.

Fix $t\in\mathbb{R}_{+}$ and an invertible $b\in\mathcal{B}$ such
that $t>1/\varepsilon$ and $\|b^{-1}\|<\varepsilon$. We use the subordination equation (\ref{eq:subordination eqn}) 
with $X_{j}$ in place of $x_{j}$ and with $B$ in place of $b$,
where
\[
B=\left[\begin{array}{cc}
b & 0\\
0 & t
\end{array}\right].
\]
 As seen above, for this particular form of $B$ we have
\[
\Omega_{j}(B)=\left[\begin{array}{cc}
\omega_{j}(b) & 0\\
\beta_{j} & t
\end{array}\right],\quad j=1,2.
\]
We proceed to determine $\beta_{j}$. The calculations above, and
one more inversion, yield
\[
G_{X}(B)=\left[\begin{array}{cc}
G_{x}(b) & 0\\
t^{-1}\mathbb{E}[y(b-x)^{-1}] & t^{-1}
\end{array}\right],\quad F_{X}(B)=\left[\begin{array}{cc}
F_{x}(b) & 0\\
-\mathbb{E}[y(b-x)^{-1}]F_{x}(b) & t
\end{array}\right].
\]
Formula (\ref{eq:don't repeat too much}) gives
\[
G_{X_{1}}(\Omega_{1}(B))=\left[\begin{array}{cc}
G_{x_{1}}(\omega_{1}(b)) & 0\\
-t^{-1}\mathbb{E}[(\beta_{1}-y)(\omega_{1}(b)-x_{1})^{-1}] & t^{-1}
\end{array}\right],
\]
and thus
\[
F_{X_{1}}(\Omega_{1}(B))=\left[\begin{array}{cc}
F_{x_{1}}(\omega_{1}(b)) & 0\\
\mathbb{E}[(\beta_{1}-y)(\omega_{1}(b)-x_{1})^{-1}]F_{x_{1}}(\omega_{1}(b)) & t
\end{array}\right].
\]
Replacing $x_{1}$ by $x_{2}$ and $y$ by zero we obtain
\[
G_{X_{2}}(\Omega_{2}(B))=\left[\begin{array}{cc}
G_{x_{1}}(\omega_{1}(b)) & 0\\
-t^{-1}\mathbb{E}[\beta_{2}(\omega_{2}(b)-x_{2})^{-1}] & t^{-1}
\end{array}\right],
\]
and 
\[
F_{X_{2}}(\Omega_{2}(B))=\left[\begin{array}{cc}
F_{x_{2}}(\omega_{2}(b)) & 0\\
\mathbb{E}[\beta_{2}(\omega_{2}(b)-x_{2})^{-1}]F_{x_{2}}(\omega_{2}(b)) & t
\end{array}\right].
\]
The $(2,1)$ entry in the equality $F_{X}(B)=F_{X_{j}}(\Omega_{j}(B))$
amounts to
\begin{align}
-\mathbb{E}[y(b-x)^{-1}]F_{x}(b) & =\mathbb{E}[(\beta_{1}-y)(\omega_{1}(b)-x_{1})^{-1}]F_{x_{1}}(\omega_{1}(b))\label{eq:finally}\\
 & =\mathbb{E}[\beta_{2}(\omega_{2}(b)-x_{2})^{-1}]F_{x_{2}}(\omega_{2}(b)),\nonumber 
\end{align}
and the last term is
\[
\mathbb{E}[\beta_{2}(\omega_{2}(b)-x_{2})^{-1}]F_{x_{2}}(\omega_{2}(b))=\beta_{2}G_{x_{2}}(\omega_{2}(b))F_{x_{2}}(\omega_{2}(b))=\beta_{2}.
\]
Thus $\beta_{2}=-\mathbb{E}[y(b-x)^{-1}]F_{x}(b)$. On the other hand,
the $(2,1)$ entry in the equality $\Omega_{1}(b)+\Omega_{2}(b)-b=F_{X}(b)$
yields
\[
\beta_{1}+\beta_{2}=-\mathbb{E}[y(b-x)^{-1}]F_{x}(b),
\]
showing that $\beta_{1}=0$. The equality $\mathbb{E}[y(b-x)^{-1}]=\mathbb{E}[y(\omega_{1}(b)-x_{1})^{-1}]$
follows now from (\ref{eq:finally}). The proof of the second identity
in the statement is similar, using upper triangular $2\times2$ matrices.
\end{proof}

\section{\label{sec: C* case}Subordination in $C^{*}$-probability spaces}

Let $(\mathcal{A},\mathcal{B},\mathbb{E})$ be an operator valued 
$C^{*}$-probability space. Thus, $\mathcal{A}$ is a unital $C^{*}$-algebra,
$\mathcal{B}\subset\mathcal{A}$ is a $C^{*}$-subalgebra containing
the unit of $\mathcal{A}$, and $\mathbb{E}:\mathcal{A}\to\mathcal{B}$
is a completely positive conditional expectation. We denote by $\mathbb{H}(\mathcal{B})\subset\mathcal{B}$
the set consisting of those elements $b\in\mathcal{B}$ whose imaginary
part $\Im b=(b-b^{*})/2i$ is nonnegative and invertible. 

Suppose now that $x_{1},x_{2}\in\mathcal{A}$ are two selfadjoint
elements that are $\mathbb{E}$-free, and set $x=x_{1}+x_{2}$.
It was shown in \cite{be-ma-spe} that there exist unique analytic functions
$\omega_{1},\omega_{2}:\mathbb{H}(\mathcal{B})\to\mathbb{H}(\mathcal{B})$
such that
\[
F_{x}(b)=F_{x_{1}}(\omega_{1}(b))=F_{x_{2}}(\omega_{2}(b))=\omega_{1}(b)+\omega_{2}(b)-b,\quad b\in\mathbb{H}(\mathcal{B}).
\]
These functions are, of course, analytic continuations of the subordination
functions considered in Section \ref{sec:Subrdination-in-Banach}.
The following statement follows from Theorem \ref{thm:main Banach algebra theorem}
by analytic continuation.
\begin{thm}
\label{thm:E(Rc)}Let $(\mathcal{A},\mathcal{B},\mathbb{E})$ be an
operator valued $C^{*}$-probability space. Consider elements $x_{1},x_{2},y\in\mathcal{A}$
such that $x_{1}$ and $x_{2}$ are selfadjoint, and $x_{2}$ is $\mathbb{E}$-free
from $\{x_{1},y\}$. Then the subordination function $\omega_{1}$
satisfies
\[
\mathbb{E}[y(b-x^{-1})]=\mathbb{E}[y(\omega_{1}(b)-x_{1})^{-1}]\ \text{ and }\ \mathbb{E}[(b-x)^{-1}y]=\mathbb{E}[(\omega_{1}(b)-x_{1})^{-1}y],
\]
for every $b\in\mathbb{H}^{+}(\mathcal{B}).$ 
\end{thm}

The result applies, for instance, to elements $y$ in the closure
of the algebra $\mathcal{B}\langle x_{1}\rangle$ generated by $\mathcal{B}$
and $x_{1}$. To see how this result relates to \cite{coalgebra}, consider
an arbitrary state $\varphi$ on $\mathcal{B},$ and construct its
extension $\psi=\varphi\circ\mathbb{E}.$ Then Theorem \ref{thm:E(Rc)}
simply states that $(b-x)^{-1}-(\omega_{1}(b)-x_{1})^{-1}$ is orthogonal
to $y^{*}$ in the scalar product induced by $\psi$. In other words,
if $\mathcal{Y}$ is any $C^{*}$-algebra algebra containing $\mathcal{B}\langle x_{1}\rangle$,
and if $\mathcal{Y}$ is $\mathbb{E}$-free from $x_{2}$, then the
orthogonal projection of $(b-x)^{-1}$ onto the closure of $\mathcal{Y}$
(in the scalar product induced by $\psi)$ is precisely $(\omega_{1}(b)-x_{1})^{-1}$.
In particular, this orthogonal projection belongs to the closure of
$\mathcal{B}\langle x_{1}\rangle$.

\end{document}